\documentclass[a4paper, 12pt, twoside]{report}
\usepackage{etex}
\usepackage[utf8]{inputenc}
\usepackage[T1]{fontenc}
\usepackage{lmodern}
\usepackage{hyphsubst}
\usepackage[english]{babel}
\usepackage{textcomp}
\usepackage{enumitem}
\usepackage{microtype}

\usepackage[onehalfspacing]{setspace}

\usepackage{listings}
\usepackage{amsmath}
\usepackage{amsfonts}
\usepackage{amssymb}
\usepackage{mathtools}
\usepackage{tikz}
\usetikzlibrary{arrows}
\usetikzlibrary{positioning}

\usepackage{graphicx}
\usepackage{caption}
\usepackage{subcaption}

\usepackage{chngcntr}
\counterwithout{equation}{chapter}

\usepackage{hyperref}
\hypersetup{
	hidelinks
}

\usepackage[figurename=Figure]{caption}
\DeclareCaptionFormat{empty}{#1} 

\usepackage{chngcntr}
\counterwithin{figure}{chapter}

\usepackage{algorithm}
\usepackage{algorithmicx}
\usepackage{algpseudocode}

\usepackage{multicol}

\newcommand{\R}{\mathbb{R}}

\newcommand{\N}{\mathbb{N}}

\newcommand{\norm}[1]{\left\lVert#1\right\rVert}
\newcommand{\abs}[1]{\left\lvert#1\right\rvert}
\DeclarePairedDelimiter{\nor}{\lVert}{\rVert}
\newcommand{\defeq}{\mathrel{\vcentcolon=}}
\newcommand{\defeqback}{\mathrel{=\vcentcolon}}

\newcommand{\mitem}{\item\(\displaystyle} 

\newcommand{\diff}{\mathrm{d}}

\providecommand\given{} 
\newcommand\SetSymbol[1][]{
   \nonscript\,#1\vert \allowbreak \nonscript\,\mathopen{}}
\DeclarePairedDelimiterX\Set[1]{\lbrace}{\rbrace}
 { \renewcommand\given{\SetSymbol[\delimsize]} #1 }

\newenvironment{definitionenum}{
\begin{enumerate}[label=\roman*)]}{
\end{enumerate}}

\newenvironment{theoremenum}{
\begin{enumerate}[label=\alph*)]}{
\end{enumerate}}

\usepackage{amsthm}
\theoremstyle{plain}
\newtheorem{theorem}{Theorem}[chapter]
\newtheorem{lemma}[theorem]{Lemma}

\theoremstyle{definition}
\newtheorem{definition}[theorem]{Definition}
\newtheorem{assumption}[theorem]{Assumption}
\theoremstyle{remark}
\newtheorem{remark}[theorem]{Remark}

\DeclareMathOperator{\Vor}{Vor}
\DeclareMathOperator{\ind}{index}
\DeclareMathOperator*{\argmin}{arg\,min}
\DeclareMathOperator{\supp}{supp}
\DeclareMathOperator{\Wass}{Wass}

\title{Bachelor Thesis}
\author{Valentin Hartmann}
\date{19.03.2016}

\begin{document}

\pagenumbering{gobble}

\begin{titlepage}
	\centering
	\includegraphics[width=\textwidth]{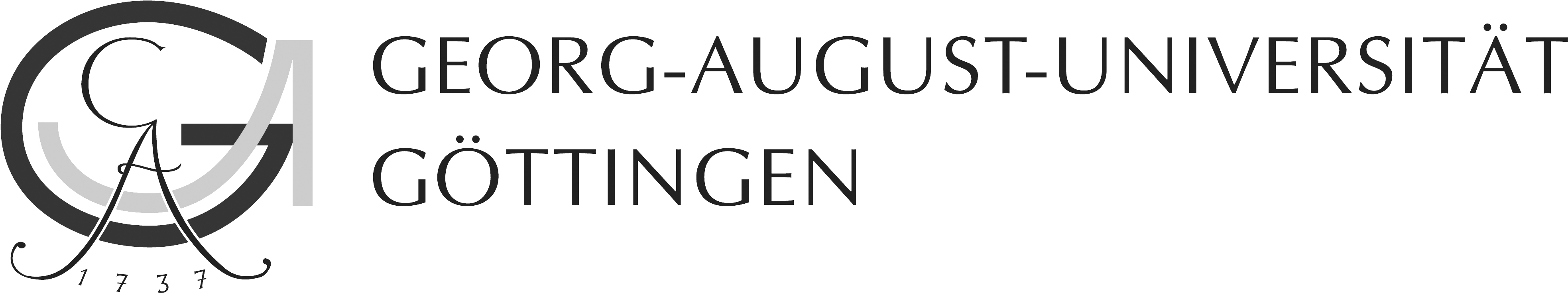}\par
	\vspace{0.5cm}
	{\scshape\large Faculty of Mathematics and Computer Science\par}
	\vspace{2cm}
	{\scshape\Large Bachelor Thesis\par}
	\vspace{2cm}
	{\bfseries\Large A Geometry-Based Approach for Solving the Transportation Problem with Euclidean Cost\par}
	\vspace{2.5cm}
	{\Large Valentin Hartmann\par}
	\vspace{2cm}
	
	\begin{multicols}{2}
	{\scshape\normalsize First Supervisor:\par}
	{\large Prof. Dr. Dominic Schuhmacher\par}
	\columnbreak
	{\scshape\normalsize Second Supervisor:\par}
	{\large Prof. Dr. Axel Munk\par}
	\end{multicols}
	
	\vspace{1cm}
	{\normalsize{\scshape Submitted:} August 19, 2016\par}
	
\end{titlepage}

\cleardoublepage

\tableofcontents

\cleardoublepage
\pagenumbering{arabic}

\chapter{Introduction}
Given a pile of sand lying on an area \(A\) and a target area \(B\), determine the cheapest way of moving the pile from A to B where each pair of a grain and a position is assigned a cost of moving said grain to said position. By raising this problem in 1781, Gaspard Monge founded the theory of optimal transport \cite{Mon81}. A measure theoretic formulation of a similar problem was later given by Kantorovich \cite{Kan42}: The source and the target pile are defined as mass distributions \(\mu\) and \(\nu\) and we are to find a distribution \(\gamma\) with marginals \(\mu\) and \(\nu\) that minimizes the cost of the overall transport,
\[\int c(x,y) \gamma(\diff x, \diff y),\]
with \(c(x,y)\) denoting the cost of the transport of one unit of mass at poisition \(x\) to position \(y\).\\
A variant that comes closer to Monge's description is the search for a map \(T\) with \(T_\# \mu = \nu\) --- that is, it transforms the mass distribution of \(\mu\) into the one of \(\nu\) --- and minimal cost
\[\int c(x,T(x)) \mu(\diff x).\]
The drawback of this version, as opposed to the one of Kantorovich, is that a map fulfilling \(T_\# \mu = \nu\) does not exist in the general case, let alone an optimal one.

The existence of an optimal map \(T\) depends on two factors: the choice of the measures \(\mu\) and \(\nu\) and the choice of the cost function \(c\). If for example \(\mu\) and \(\nu\) are discrete with finite support, the problem turns into a combinatorial optimization problem. In this thesis we are going to study the case where \(\mu\) is a continuous and \(\nu\) is a discrete probability measure. In this setting existence of a solution to the transportation problem can be shown for a large class of cost functions using mainly geometric arguments \cite{GKPR13}. The cost functions include arbitrary powers \(p\) of the Euclidean distance \(\norm{x-y}^p\). We will examine the case where \(p=1\).\\
One can visualize the problem as follows. Consider a large town with several hospitals spread across it, each providing a limited number of beds. In addition, the town is split into different areas for each of which the population density is known. The municipality wants to hand out a plan to its citizens telling them which hospital to visit in case of an emergency based on their place of residence. The assignment should be done in such a way that the bed occupation rate is the same for each hospital and that the average journey for the visitors is as short as possible.

For \(p=2\) an algorithm to explicitly compute the map \(T\) has been developed by Aurenhammer et al. \cite{AHA98} and recently been improved in speed by Mérigot \cite{Mer11} by using a decomposition of \(\nu\). We base our work on Mérigot's idea of turning the problem into a convex optimization problem but adjust the algorithm to work for \(p=1\) and give a novel proof of the soundness of this approach.

\chapter{Preliminaries}
Before we can begin looking at the problem of optimal transport, we have to introduce a few concepts and notations. Nevertheless, the reader should already be familiar with general measure theory since we will not start from ground up.

\section{Discrete and Continuous Measures}
In the following we will always equip \(\R^d\) with the Borel \(\sigma\)-algebra.

\begin{definition}
A \emph{measure} \(\gamma\) on \(\R^d\) is a map from the measurable subsets of \(\R^d\) to \([0,\infty]\) with the following properties:
\begin{definitionenum}
\mitem \gamma(\emptyset) = 0\)
\mitem \gamma(\bigcup_{i=1}^\infty A_i) = \sum_{i=1}^\infty\gamma(A_i)\quad\)for all countable families \((A_i)_i\) of pairwise disjoint measurable sets (\emph{\(\sigma\)-additivity})
\end{definitionenum}
We call \(\gamma\) a \emph{probability measure} if \(\gamma(\R^d)=1\).
\end{definition}

\begin{definition}
A measure \(\gamma\) is called
\begin{itemize}
	\item \emph{continuous} if it has a density, that is, a non-negative measurable function \(\rho\) such that \[\gamma(A) = \int_A\rho(x)\diff x\] for all measurable sets \(A\subset\R^d\);
	\item \emph{discrete} if there exist a countable set S of points in \(\R^d\) and positive numbers \((\lambda_p)_{p\in S}\) such that \[\gamma=\sum_{p\in S}\lambda_p\delta_p\] where \(\delta_p\) is the probability measure with \(\delta_p(\{p\})=1\).
\end{itemize}
\end{definition}

\begin{assumption}
\label{ass_measures}
In our case, both the source measure \(\mu\) and the target measure \(\nu\) are probability measures. \(\mu\) is continuous with density \(\rho\) which, according to the Radon-Nikodym theorem, implies absolute continuity with respect to the Lebesgue measure, that is, it vanishes at Lebesgue null set. In addition, we require \(\int_{\R^d} \norm{x} \mu(\diff x)<\infty\). We will see how this comes into play in the next chapter.\\
\(\nu\) is a discrete measure supported on the finite set \(S = \{s^1, \dots, s^n\}\) with masses \(\lambda_i\), \(v = \sum_i\lambda_i\delta_{s^i}\).
\end{assumption}

\begin{definition}
\label{def:pushforward}
Let \(T:E_1\to E_2\) be a measurable map between the measure space \((E_1, \mathcal{E}_1, \gamma_1)\) and the measurable space \((E_2, \mathcal{E}_2)\) where \(T^{-1}(A)\in\mathcal{E}_1\) for all \(A\in\mathcal{E}_2\). The measure \(T_\#\gamma_1\) on \((E_2, \mathcal{E}_2)\), defined by
\[T_\#\gamma_1(A) \defeq \gamma_1(T^{-1}(A)) \quad \text{for all } A \in \mathcal{E}_2,\]
is called \emph{pushforward} of \(\gamma_1\) with respect to \(T\).\\
\(T\) is said to be a \emph{transport map} from \(\gamma_1\) to \(T_\#\gamma_1\).
\end{definition}

\section{Optimal Transport}
\label{ssec:Voronoi}
This thesis aims to find a transport map from the source measure \(\mu\) to the target measure \(\nu\) that has minimal cost in the following sense:
\begin{definition}
The \emph{cost} of a transport map \(T\) from \(\mu\) to \(\nu\) is defined by
\[c(T) \defeq \int_{\R^d}\norm{x-T(x)}\mu(\diff x) = \sum_{p\in S}\int_{T^{-1}(p)}\norm{x-p}\mu(\diff x)\]
where \(\norm{\cdot}\) denotes the Euclidean norm.
\end{definition}

\begin{remark}
Note that \(c(T)\) is bounded from above by a constant \(C < \infty\) for every transport map \(T\):
\begin{align}
c(T) &= \int_{\R^d}\norm{x-T(x)}\mu(\diff x)\nonumber\\
&\leq \int_{\R^d}\norm{x}\mu(\diff x) + \int_{\R^d}\norm{T(x)}\mu(\diff x)\nonumber\\
&\leq \int_{\R^d}\norm{x}\mu(\diff x) + \max_{p\in S} \lVert p\rVert \int_{\R^d}\mu(\diff x)\nonumber\\
&= \int_{\R^d}\norm{x}\mu(\diff x) + \max_{p\in S} \lVert p\rVert\nonumber\\
&\defeqback C < \infty \label{ineq:C}
\end{align}
where from the first to the second line we used the triangle inequality.
\end{remark}

\begin{definition}
The Wasserstein distance \(\Wass(\mu, \nu)\) between \(\mu\) and \(\nu\) is given by \(\sqrt{c(U)}\) where \(U\) is a transport map from \(\mu\) to \(\nu\) with minimal cost.
We refer to such a map as an \emph{optimal transport map}.
\end{definition}
We can think of the Wasserstein distance as a value that indicates how far the masses of \(\mu\) and \(\nu\) are separated from each other in \(\R^d\).\\
In \hyperref[sec:Opt_Trans]{Chapter \ref*{sec:Opt_Trans}} we will show that an optimal transport map as above always exists. To construct it explicitly, we need a geometric object called \emph{weighted Voronoi diagram}.

\section{Additively Weighted Voronoi Diagrams}
\begin{definition}
\label{def:voronoi}
Let \(M = \{m_1, \dots, m_k\}\) be a finite set of point sites in \(\R^d\) and \(w\in\R^k\) a so called \emph{weight vector} that assigns to each site \(m_i\) a weight \(w_i\).
The \emph{Voronoi cell} of the site \(m_i\) is defined as
\[\Vor_M^w(i) \defeq \Set{x\in\R^d \given \norm{x-m_i} - w_i \leq \norm{x-m_j} - w_j \quad\text{for all } j \neq i}.\]
The \emph{additively weighted Voronoi diagram} is the collection \(\left(\Vor_M^w(i)\right)_i\) of all Voronoi cells. Note that the interior of the cells together with the union of their boundaries form a partition of \(\R^d\).
\end{definition}

By this definition it becomes clear that the Voronoi diagram does not change if we add the same constant to each entry of the weight vector.

A more intuitive definition is the following: Imagine each site \(m_i\) as surrounded by a ball with radius \(w_i\). Then \(\Vor_M^w(i)\) contains exactly the points that are not closer to any other ball than to that around the i-th site. Here, points lying inside a ball have a negative distance to said ball. Therefore, sites whose balls are contained completely in another ball have an empty Voronoi cell. In the two-dimensional case the boundary between two non-empty cells is a straight line if their weights are equal and a section of a hyperbola otherwise.

\begin{figure}
	\includegraphics[width=\textwidth]{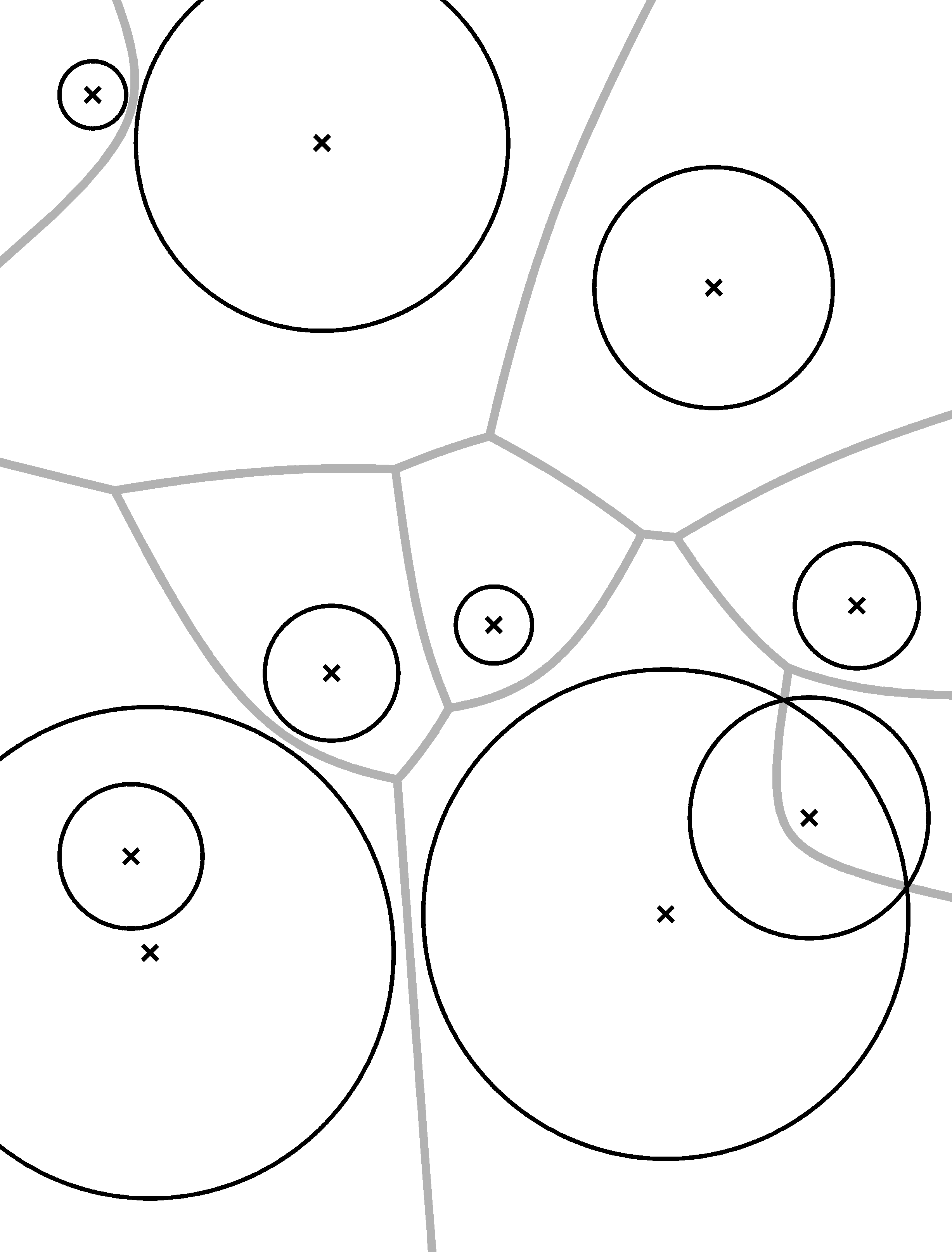}
	\caption{A Voronoi diagram with 10 sites in \(\R^2\).}
\end{figure}

\begin{definition}
\label{def:Voronoi_Transport}
Given such a Voronoi diagram, we denote the function that maps each point in \(\R^d\) to the center of the Voronoi cell containing it by \(T_M^w\). For its well-definedness see the remark below.
\end{definition}

\begin{remark}
\(T_M^w\) is \(\mu\)-almost everywhere well-defined; the boundaries of the Voronoi cells on which it is not well-defined are null sets. Together with the continuity result from \hyperref[lemma:continuity]{Lemma \ref*{lemma:continuity}} this follows from Lemma 1 in \cite{GKPR13}. Note that for \(d=2\) it can also be seen from \hyperref[ssec:Drawing]{Section~\ref*{ssec:Drawing}} where we show that in this case the boundaries are one-dimensional manifolds.
\end{remark}

\chapter{Finding the Optimal Transport Map}
\label{sec:Opt_Trans}

In this chapter we show that an optimal transport map from \(\mu\) to \(\nu\) always exists and that it is unique. We obtain this transport map by creating an additively weighted Voronoi diagram around the points in the support \(S\) of \(\nu\) and assigning each point in \(\R^d\) to the center of the Voronoi cell containing it. In order to get a valid transport map from a Voronoi diagram, we need the notion of an \emph{adapted weight vector}.

\begin{definition}
A weight vector \(w\) as in \hyperref[def:voronoi]{Definition \ref*{def:voronoi}} is called \emph{adapted} to the pair \((\mu, \nu)\) if
\[\mu(\Vor_S^w(i)) = \lambda_i \ [=\nu(\{s^i\})]\]
for every \(i\in\{1, \dots, n\}\).\\
By \hyperref[def:pushforward]{Definition \ref*{def:pushforward}}, the condition can be rewritten as
\[\left(T_S^{w}\right)_\# \mu = \nu\]
since
\[\left(\left(T_S^{w}\right)_\#\mu\right)(\{s^i\}) = \mu\left(\left(T_S^{w}\right)^{-1}(\{s^i\})\right) = \mu(\Vor_S^{w}(i)).\]
\end{definition}

We begin with the uniqueness result whose proof follows that of Theorem~2 in \cite{GKPR13}.

\begin{theorem}
\label{thm:uniqueness}
For every \(w\in\R^n\), \(T_S^{w}\) is the \(\mu\)-almost surely unique optimal transport map from \(\mu\) to \(\left(T_S^{w}\right)_\# \mu\).\\
In other words, if \(w\) is adapted to \((\mu, \nu)\), then \(T_S^{w}\) is the \(\mu\)-almost surely unique optimal transport map from \(\mu\) to \(\nu\).
\end{theorem}
\begin{proof}
Let \(w\in\R^n\) adapted to \((\mu, \nu)\). For simplicity, write \(T^*\defeq T_S^w\) and let \(T\) be a different transport map from \(\mu\) to \(\nu\) that differs from \(T^*\) on a set \(M\defeq \Set{x\given T(x)\neq T^*(x)}\) of positive measure, that is, \(\mu(M)>0\). By removing the boundaries of the Voronoi cells from \(M\), we get a set \(\tilde{M}\) which still has positive measure since the boundaries are null sets. Then, by definition of the Voronoi diagram, we find that for each \(x\in\tilde{M}\)
\[\norm{x-T^*(x)} - w_{T^*(x)} < \norm{x-T(x)} - w_{T(x)}\]
which leads to
\begin{equation}
\label{eq:ineq_int_transport}
\int_{\R^n} \left(\norm{x-T^*(x)} - w_{T^*(x)}\right) \mu(\diff x) < \int_{\R^n} \left(\norm{x-T(x)} - w_{T(x)}\right) \mu(\diff x).
\end{equation}
To see this, let
\[f(x)\defeq \left(\norm{x-T(x)} - w_{T(x)}\right) - \left(\norm{x-T^*(x)} - w_{T^*(x)}\right).\]
\(f\) is positive on \(\tilde{M}\) and 0 everywhere else, therefore,
\[\tilde{M} = \Set{x\given f(x) \geq 1} \cup \left(\bigcup_{k\in\N} \Set[\bigg]{x\given\frac{1}{k}>f(x)\geq\frac{1}{k+1}}\right).\]
By the \(\sigma\)-additivity of \(\mu\) we get that
\begin{align*}
0 < \mu(\tilde{M}) &= \mu\left(\Set{x\given f(x) \geq 1} \cup \left(\bigcup_{k\in\N} \Set[\bigg]{x\given\frac{1}{k}>f(x)\geq\frac{1}{k+1}}\right)\right)\\
&= \mu\left(\Set{x\given f(x) \geq 1}\right) + \sum_{k\in\N} \mu\left( \Set[\bigg]{x\given\frac{1}{k}>f(x)\geq\frac{1}{k+1}}\right),
\end{align*}
meaning that at least one of the sets partitioning \(\tilde{M}\) has positive measure. Hence, there exists an \(m\in\N\) such that \(\mu\left(\Set{x\given f(x)\geq\frac{1}{m}}\right) > 0\), implying
\[\int_{\R^n} f(x)\mu(\diff x) \geq \int_{\Set{x\given f(x)\geq\frac{1}{m}}} \frac{1}{m}\mu(\diff x) > 0\]
which proves \eqref{eq:ineq_int_transport}.
Now we can use \eqref{eq:ineq_int_transport} to obtain the following inequality:
\begin{align*}
c(T^*) &= \int_{\R^n} \norm{x-T^*(x)} \mu(\diff x)\\
&= \int_{\R^n} \left(\norm{x-T^*(x)} - w_{T^*(x)} + w_{T^*(x)}\right)\mu(\diff x)\\
&= \int_{\R^n} \left(\norm{x-T^*(x)} - w_{T^*(x)}\right)\mu(\diff x) + \int_{\R^n} w_{T^*(x)}\mu(\diff x)\\
&< \int_{\R^n} \left(\norm{x-T(x)} - w_{T(x)}\right)\mu(\diff x) + \int_{\R^n} w_{T^*(x)}\mu(\diff x)\\
&= c(T) - \int_{\R^n} w_{T(x)}\mu(\diff x) + \int_{\R^n} w_{T^*(x)}\mu(\diff x)
\end{align*}
Both \(T\) and \(T^*\) are transport maps: \(T_\#\mu = \nu = T_\#^*\mu\). Thus,
\begin{align*}
\int_{\R^n} w_{T(x)}\mu(\diff x) &= \sum_{i=1}^n w_i T_\#\mu(\{s^i\})\\
&= \sum_{i=1}^n w_i T_\#^*\mu(\{s^i\})\\
&= \int_{\R^n} w_{T^*(x)}\mu(\diff x)
\end{align*}
and
\[c(T) - \int_{\R^n} w_{T(x)}\mu(\diff x) + \int_{\R^n} w_{T^*(x)}\mu(\diff x) = c(T).\]
\end{proof}

Note that not just the optimal transport map is only almost surely unique but also the Voronoi diagram we obtain \(T_S^w\) from; meaning there might be other Voronoi diagrams that produce the same transport map. Indeed, if the support of \(\rho\) is not pathwise connected, then moving the boundaries of a cell in the gap between two connected components only results in a change of the transport map on a set of measure zero and thus leaves its cost unaffected.\\
For the case of a pathwise connected support, however, uniqueness has been shown by Geiß et al. in \cite[Section 5]{GKPR13}.

The existence theorem requires that the volume of a Voronoi cell changes continuously as its boundary sweeps across the space. This is formalized in the following lemma.

\begin{lemma}
\label{lemma:continuity}
The map
\[w \mapsto \mu(\Vor_S^w(i))\]
is continuous.
\end{lemma}

\begin{proof}
Let \(w\) fixed and let
\[A_m\defeq\Vor_S^{w+\frac{1}{m}e_i}(i), \quad B_m\defeq\Vor_S^{w-\frac{1}{m}e_i}(i).\]
By definition of \(\Vor_S^w(i)\) it holds that \(A_{m+1} \subset A_m\) and \(B_{m} \subset B_{m+1}\).\\
We first show that
\[\bigcap_m A_m = \Vor_S^w(i) \quad\text{and}\quad \bigcup_m B_m = \Vor_S^w(i)^\circ\]
where \(\Vor_S^w(i)^\circ\) denotes the interior of the Voronoi cell \(\Vor_S^w(i)\).\\
Let \(x\notin\Vor_S^w(i)\). This means that there exists a \(k\in\{1, \dots, n\}\setminus\{i\}\) such that
\[d_k^i\defeq \lVert x-s^k\rVert - w(k) - \lVert x-s^i\rVert + w(i) < 0.\]
Since \(\frac{1}{m}\rightarrow 0\), we can find an \(M\in\N\) with \(d_k^i + \frac{1}{m} < 0\) for all \(m\geq M\). This implies \(x\notin A_m\) for all \(m\geq M\). We further observe that \(\Vor_S^w(i) \subset A_m\) for all~\(m\). Hence, \(\bigcap_m A_m = \Vor_S^w(i)\).\\
Now let \(x\in\Vor_S^w(i)^\circ\). With the notation from above this is equivalent to \(d_i^k < 0\) for all \(k\neq i\) (notice that \(i\) and \(k\) switched positions). Again, there exists an \(M\in\N\) such that \(d_i^k + \frac{1}{m} < 0\) for all \(m\geq M\) and all \(k\neq i\). So \(x\in B_m\) for all \(m\geq M\). Obviously, \(B_m \subset \Vor_S^w(i)^\circ\) for all \(m\). This proves the claim \(\bigcup_m B_m = \Vor_S^w(i)^\circ\).

Let \(\varepsilon>0\). With the continuity of \(\mu\) from above we find an \(M_A\in\N\) such that
\[\abs{\mu(\Vor_S^w(i)) - \mu(A_m)}<\varepsilon \quad\text{for all } m\geq M_A\]
and with the continuity of \(\mu\) from below we find an \(M_B\in\N\) such that
\[\abs{\mu(\Vor_S^w(i)) - \mu(B_m)} \stackrel{(*)}{=} \abs{\mu(\Vor_S^w(i)^\circ) - \mu(B_m)}<\varepsilon \quad\text{for all } m\geq M_B\]
where (*) holds true because the boundary of a Voronoi cell is a null set.\\
These two inequalities imply
\begin{equation}
\label{ineq_size_cell}
\abs{\mu(\Vor_S^w(i)) - \mu(\Vor_S^{w+re_i}(i))}<\varepsilon \quad\text{for all } \abs{r}\leq \frac{1}{\max\{M_A, M_B\}}
\end{equation}
because
\[B_{\left\lfloor\frac{1}{\abs{r}}\right\rfloor} = \Vor_S^{w-\left\lfloor\frac{1}{\abs{r}}\right\rfloor^{-1}e_i}(i) \subseteq \Vor_S^{w-\abs{r}e_i}(i) \subseteq \Vor_S^w(i)\]
and
\[\Vor_S^w(i) \subseteq \Vor_S^{w+\abs{r}e_i}(i) \subseteq \Vor_S^{w+\left\lfloor\frac{1}{\abs{r}}\right\rfloor^{-1}e_i}(i) = A_{\left\lfloor\frac{1}{\abs{r}}\right\rfloor}.\]
Let \(v\in\R^n\) and \(v_\text{max}\defeq\max_k\abs{v_k}\leq\norm{v}\). We recall the definition of a Voronoi cell \(\Vor_S^u(i)\): A point \(x\) belongs to \(\Vor_S^u(i)\) if and only if
\[\lVert x-s^i\rVert - \lVert x-s^j\rVert \leq u_i - u_j\quad \text{for all } j\in \{1, \dots, n\}.\]
Therefore, it follows from
\[(w_i - 2v_\text{max}) - w_j \leq (w_i + v_i) - (w_j + v_j) \leq (w_i + 2v_\text{max}) - w_j\]
that
\[\Vor_S^{w-2v_\text{max}e_i}(i) \subseteq \Vor_S^{w+v}(i) \subseteq \Vor_S^{w+2v_\text{max}e_i}(i)\]
and hence
\begin{equation*}
\begin{split}
\abs{\mu(\Vor_S^w(i)) - \mu(\Vor_S^{w+v}(i))} \leq
\max\left\{\abs{\mu(\Vor_S^w(i)) - \mu(\Vor_S^{w-2v_\text{max}e_i}(i))}\right.,\\
\left.\abs{\mu(\Vor_S^w(i)) - \mu(\Vor_S^{w+2v_\text{max}e_i}(i))}\right\}.
\end{split}
\end{equation*}
This yields by \eqref{ineq_size_cell}
\[\abs{\mu(\Vor_S^w(i)) - \mu(\Vor_S^{w+v}(i))} < \varepsilon \quad\text{for all } \norm{v}\leq \frac{1}{2\max\{M_A, M_B\}}\]
which concludes the proof.
\end{proof}

It is not clear from the definition that a weight vector adapted to \((\mu,\nu)\) can be found in the general case. By using the same idea as \cite{Mer11} of transforming the problem into a convex optimization problem, we not only get a proof of existence but also an explicit construction that forms the basis for our algorithm in \hyperref[sec:Algorithm]{Chapter \ref*{sec:Algorithm}}.\\
The function \(\Phi\) we are optimizing is given in the following theorem.

\begin{theorem}
\label{thm:Phi}
Let \(\Phi : \R^n \to \R\),
\[\Phi(w) \defeq \sum_{i=1}^n\left(-\lambda_i w_i - \int_{\Vor_S^w(i)} \left(\lVert x - s^i\rVert - w_i\right)\mu(\diff x)\right).\]
Then:
\begin{theoremenum}
\item \(\Phi\) is convex.
\label{enum:phi1}
\item \(\Phi\) is differentiable with partial derivatives
\[\frac{\partial\Phi}{\partial w_i}(w) = -\lambda_i+\mu(\Vor_S^w(i)).\]
\label{enum:phi2}
\end{theoremenum}
Note that \(\Phi\) can be written as
\[\Phi(w) = \sum_i(-\lambda_i w_i) - \int_{\R^d} (\norm{x - T_S^w(x)} - w_{T_S^w(x)})\mu(\diff x)\]
where we use the notation \(w_{s^i}\defeq w_i\) for convenience.
\end{theorem}

\begin{proof}
We will examine the non-linear part \(\Psi\) of \(\Phi\):
\begin{align*}
\Phi(w) &= \sum_i (-\lambda_i w_i) - \Psi(w), \quad \text{where}\\
\Psi(w) &\defeq \int_{\R^d} (\norm{x - T_S^w(x)} - w_{T_S^w(x)})\mu(\diff x)\\
&= \sum_i\int_{\Vor_S^w(i)} \left(\lVert x - s^i\rVert - w_i\right)\mu(\diff x).
\end{align*}\\
We show that \(\Psi\) is concave and therefore \(-\Psi\) is convex. Statement \ref{enum:phi1} then follows from the fact that linear functions are convex which makes \(\Phi\) the sum of two convex functions.

To this end, let \(F\defeq\Set{f:\R^d\to S \given f\ \text{measurable}}\) be the set of all measurable maps from \(\R^d\) to \(S\). The definition of \(T_S^w\) yields that for every \(f\in F\) and for every \(x\in\R^d\) the inequality
\[\norm{x - T_S^w(x)} - w_{T_S^w(x)} \leq \norm{x - f(x)} - w_{f(x)}\]
holds since \(T_S^w\) maps each point \(x\) to its nearest neighbor in \(S\) in terms of the Voronoi diagram. Hence,
\begin{equation}
\label{eq:ineq_psi_f}
\Psi(w) \leq \Psi_f(w) \defeq \int_{\R^d} (\norm{x - f(x)} - w_{f(x)})\mu(\diff x)
\end{equation}
with equality for \(f = T_S^w\) which allows us to write \(\Psi\) as
\[\Psi(w) = \inf_{f\in F} \Psi_f(w).\]
Since all \(\Psi_f\) are affinely linear in \(w\), they are concave. This makes \(\Psi\) the infimum of a set of concave functions and therefore concave. Indeed, using the concavity of \(\Psi_f\) for the first inequality below, we get for arbitrary \(v, w \in \R^n\) and \(\alpha\in [0,1]\):
\begin{align*}
\Psi(\alpha v + (1-\alpha) w) &= \inf_f \Psi_f(\alpha v + (1-\alpha) w)\\
&\geq \inf_f \left(\alpha \Psi_f(v) + (1-\alpha) \Psi_f(w)\right)\\
&\geq \alpha \inf_f \Psi_f(v) + (1-\alpha) \inf_f \Psi_f(w)\\
&= \alpha \Psi(v) + (1-\alpha) \Psi(w)
\end{align*}

In order to prove statement \ref{enum:phi2}, we show that the partial derivatives of \(\Psi\) exist and that
\[\frac{\partial\Psi}{\partial w_i}(w) = -\mu(\Vor_S^w(i))\]
because
\[\frac{\partial\Phi}{\partial w_i}(w) = -\lambda_i-\frac{\partial\Psi}{\partial w_i}.\]
Let \(h\neq0\) and \(\tilde{w} \defeq w+he_i\). Inequality \eqref{eq:ineq_psi_f} implies

\begin{align*}
\frac{\Psi(\tilde{w}) - \Psi(w)}{h} &= \frac{\Psi_{T_S^{\tilde{w}}}(\tilde{w}) - \Psi(w)}{h}
\geq \frac{\Psi_{T_S^{\tilde{w}}}(\tilde{w}) - \Psi_{T_S^{\tilde{w}}}(w)}{h},\\
\frac{\Psi(\tilde{w}) - \Psi(w)}{h} &= \frac{\Psi(\tilde{w}) - \Psi_{T_S^{w}}(w)}{h}
\leq \frac{\Psi_{T_S^{w}}(\tilde{w}) - \Psi_{T_S^{w}}(w)}{h}.
\end{align*}
If we combine the two lines, we get
\begin{equation*}
\begin{split}
\abs{\frac{\Psi(\tilde{w}) - \Psi(w)}{h} + \mu(\Vor_S^w(i))}
\leq \max\left\{\abs{\frac{\Psi_{T_S^{\tilde{w}}}(\tilde{w}) - \Psi_{T_S^{\tilde{w}}}(w)}{h} + \mu(\Vor_S^w(i))}\right., \\
\left.\abs{\frac{\Psi_{T_S^{w}}(\tilde{w}) - \Psi_{T_S^{w}}(w)}{h} + \mu(\Vor_S^w(i))}\right\}.
\end{split}
\end{equation*}
We may rewrite \(\Psi_{T_S^u}(v)\) as
\[\Psi_{T_S^u}(v) = \sum_i\left(\int_{\Vor_S^u(i)}\lVert x - s^i\rVert\mu(dx) - v(i)\mu(\Vor_S^u(i))\right)\]
and therefore obtain
\begin{align*}
\abs{\frac{\Psi_{T_S^{\tilde{w}}}(\tilde{w}) - \Psi_{T_S^{\tilde{w}}}(w)}{h} + \mu(\Vor_S^w(i))}
&= \abs{-\mu(\Vor_S^{\tilde{w}}(i)) + \mu(\Vor_S^w(i))}, \\
\abs{\frac{\Psi_{T_S^{w}}(\tilde{w}) - \Psi_{T_S^{w}}(w)}{h} + \mu(\Vor_S^w(i))} &= 0.
\end{align*}
This yields
\[\abs{\frac{\Psi(\tilde{w}) - \Psi(w)}{h} + \mu(\Vor_S^w(i))} \leq \abs{-\mu(\Vor_S^{\tilde{w}}(i)) + \mu(\Vor_S^w(i))} \xrightarrow{h \to 0} 0\]
which proves the statement about the partial derivatives. We used here the continuity of \(w \mapsto \mu(\Vor_S^w(i))\) from the above lemma which also directly implies the continuity of the partial derivatives and hence the total differentiability of \(\Psi\) and \(\Phi\).
\end{proof}

\begin{remark}
For any \(w^* \in \argmin_w \Phi(w)\) \hyperref[thm:Phi]{Theorem \ref*{thm:Phi}} yields that
\[-\lambda_i+\mu(\Vor_S^{w^*}(i)) = 0 \quad\text{for all } i\in\{1, \dots, n\}\]
since the gradient of \(\Phi\) vanishes at such a \(w^*\). Thus, \(w^*\) is adapted to \((\mu, \nu)\) and \hyperref[thm:uniqueness]{Theorem \ref*{thm:uniqueness}} assures us that \(T_S^{w^*}\)  is the \(\mu\)-almost surely unique optimal transport map from \(\mu\)~to~\(\nu\). But for the existence of this induced map we still require that
\[\argmin_{w\in\R^n} \Phi(w) \neq \emptyset.\]
\end{remark}

\begin{theorem}
\(\Phi\) reaches its infimum at a weight vector \(w\in\R^n\).
\end{theorem}
\begin{proof}
Let \((w^k)_{k\in\N}\), \(w^k\in\R^n\), a sequence with
\[\lim_{k\to\infty} \Phi(w^k) = \inf_{w\in\R^n} \Phi(w).\]
Our goal is to find a bounded subsequence \((v^k)\) of \((w^k)\). Then, by the Bolzano-Weierstrass theorem, there exists a convergent subsequence \((u^k)\) of \((v^k)\) which, as it is a subsequence of \((w^k)\), still satisfies \(\lim_{k\to\infty} \Phi(u^k) = \inf_{w\in\R^n} \Phi(w)\). Since \(\Phi\) is continuous, it thus takes its infimum at \(u\defeq\lim_{k\to\infty} u^k \in\R^n\).

First note that adding the same constant to each of the entries of the argument of \(\Phi\) leaves its value unchanged: For arbitrary \(r\in\R\) and \(c\defeq(r, \dots, r) \in\R^n\) we have
\begin{align*}
&\Phi(w+c)\\
&= \sum_i\left(-\lambda_i (w_i+r) - \int_{\Vor_S^{w+c}(i)} \left(\lVert x - s^i\rVert - (w_i+r)\right)\mu(\diff x)\right)\\
&= \sum_i\left(-\lambda_i (w_i+r) - \int_{\Vor_S^{w}(i)} \lVert x - s^i\rVert \mu(\diff x) + (w_i+r)\mu(\Vor_S^{w}(i))\right)\\
&= \sum_i\left(-\lambda_i w_i - \int_{\Vor_S^{w}(i)} \lVert x - s^i\rVert \mu(\diff x) + w_i\mu(\Vor_S^{w}(i))\right)\\
&\phantom{{}={}}+ r\left(\sum_i-\lambda_i + \sum_i\mu(\Vor_S^{w}(i))\right)\\
&= \Phi(w)
\end{align*}
where in the last step we used that both \(\mu\) and \(\nu\) are probability measures and therefore their masses add up to \(1\).\\
For this reason, we may assume without loss of generality that \(w_i^k\geq 0\) for all \(i\) and \(k\).

Since \((w^k)_k\) is a sequence in \(\R^n\), each \(w^k\) only has a finite number of entries and thus there exists at least one entry \(i\) and an associated infinite set \(K\subset\N\) such that \(w_i^k \geq w_j^k\) for all \(j\in\{1, \dots, n\}\) and all \(k\in K\), giving us our first subsequence \((w^k)_{k\in K}\).\\
We are now going to iteratively create refined subsequences of \((w^k)_{k\in K}\). Consider the sequences \((w_i^k)_{k\in K}\) and \((w_1^k)_{k\in K}\). There either exists a subsequence \((w^l)_{l\in L_1\subset K}\) fulfilling \(w_i^l-w_1^l\leq R_1\) for a constant \(R_1\geq 0\) and all \(l\in L_1\) or a subsequence such that \(w_i^l - w_1^l\geq \ind(l)\) for all \(l\in L_1\), or both. Here, \(\ind(l)\) denotes the index of \(l\) in the ordered set \(L_1\). Choose the subsequence that exists and apply the same scheme on \((w_2^l)_{l\in L_1}\), continuing until we reach \((w_n^l)_{l\in L_{n-1}}\).
At the end we are left with a set \(A\subset \{1, \dots, n\}\) with \(i\in A\), a constant \(R \geq 0\) and a subsequence \((w^l)_{l\in L\defeq L_n}\) with the following properties:
\begin{enumerate}[label=\roman*)]
	\mitem 0 \leq w_i^l - w_j^l \leq R \quad \text{for all } j \in A\) \label{enum:subsequence1}
	\mitem w_i^l - w_j^l \geq \ind(l) \quad \text{for all } j \notin A\) \label{enum:subsequence2}
\end{enumerate}
For all \(l\in L\) let
\[w_{\text{min}_A}^l \defeq \min_{j\in A} w_j^l.\]

If \(A=\{1, \dots, n\}\), define the sequence \((v^l)_{l\in L}\) via \(v^l \defeq w^l - (w_{\text{min}_A}^l, \dots, w_{\text{min}_A}^l)\).\\
By property \ref{enum:subsequence1} we have that
\[0 \leq v_j^l \leq R \quad\text{for all } j \in \{1, \dots, n\} \text{ and all } l \in L\]
which makes \((v^l)_{l\in L}\) the bounded sequence we were searching for.

Now consider the case where \(B \defeq \{1, \dots, n\} \backslash A \neq \emptyset\). \ref{enum:subsequence2} implies that \(\lVert w^l\rVert\to\infty\) as \(l\to\infty\), as well as the existence of an \(N\in L\) such that
\[\bigcup_{a\in A} \Vor_S^{w^l}(a) = \R^d \quad\text{for all } l \geq N\]
since at some point all \(w_b^l\)-balls around the sites \(b\in B\) will be completely contained in the \(w_a^l\)-balls around the sites \(a\in A\). Therefore,
\begin{equation}
\label{eq:sum_A}
\sum_{a\in A} \mu\left(\Vor_S^{w^l}(a)\right) = 1 \quad\text{for all } l \geq N.
\end{equation}
Analogously to \(w_{\text{min}_A}^l\), let
\[w_{\text{max}_B}^l \defeq \max_{j\in B} w_j^l\]
and let
\[C_l \defeq \sum_{j=1}^n \int_{\Vor_S^{w^l}(j)}\lVert x-s^j\rVert\mu(\diff x) \leq C\]
for the constant \(C<\infty\) from \eqref{ineq:C}.\\
With this notation we get for arbitrary \(l\in L\) where \(l\geq N\) that
\begin{align*}
&\Phi\left(w^l\right)\\
&= \sum_{j=1}^n\left(-\lambda_j w_j^l - \int_{\Vor_S^{w^l}(j)} \left(\lVert x - s^j\rVert - w_j^l\right)\mu(\diff x)\right)\\
&= \sum_{j=1}^n -\int_{\Vor_S^{w^l}(j)} \lVert x-s^j\rVert \mu(\diff x) + \sum_{j=1}^n w_j^l\left(-\lambda_j + \mu\left(\Vor_S^{w^l}(j)\right)\right)\\
&= -C_l + \sum_{j\in A} w_j^l\left(-\lambda_j + \mu\left(\Vor_S^{w^l}(j)\right)\right) + \sum_{j\in B} w_j^l\left(-\lambda_j + \mu\left(\Vor_S^{w^l}(j)\right)\right)\\
&\geq -C_l + \sum_{j\in A} w_j^l\left(-\lambda_j + \mu\left(\Vor_S^{w^l}(j)\right)\right) - w_{\text{max}_B}^l \sum_{j \in B} \lambda_j\\
&\geq -C + w_{\text{min}_A}^l \left(1 - \sum_{j \in A} \lambda_j\right) - R - w_{\text{max}_B}^l \sum_{j \in B} \lambda_j\\
&= -C + w_{\text{min}_A}^l \sum_{j \in B} \lambda_j - R - w_{\text{max}_B}^l \sum_{j \in B} \lambda_j\\
&= -C - R + \sum_{j \in B} \lambda_j\left(w_{\text{min}_A}^l - w_{\text{max}_B}^l\right)\\
&\geq -C - R + \sum_{j \in B} (\ind(l) - R).
\end{align*}
For the second inequality we used \ref{enum:subsequence1} and \eqref{eq:sum_A} and for the factor \(-1\) in front of \(R\) the fact that
\[\sum_{j\in D} \left(-\lambda_j + \mu\left(\Vor_S^{w^l}(j)\right)\right) = -\nu(D) + \mu\left(\bigcup_{j\in D} \left(\Vor_S^{w^l}(j)\right)\right) \geq -1\]
for all \(D\subset\{1, \dots, n\}\) since both \(\nu\) an \(\mu\) are probability measures. The last inequality is obtained with the help of \ref{enum:subsequence1} and \ref{enum:subsequence2}.\\
Thus, \(\lim_{l\to\infty} \Phi(w^l) = \infty\) since \(\ind(l) \to \infty\) for \(l\to\infty\). But that contradicts our initial assumption that \(\lim_{k\to\infty} \Phi(w^k) = \inf_{w\in\R^n} \Phi(w)\). Therefore, the second case does not occur and we are always in the first case where \(A=\{1, \dots, n\}\) and for which we were able to construct a bounded sequence \((\nu^l)\) such that \(\lim_{l\to\infty} \Phi(v^l) = \inf_{w\in\R^n} \Phi(w)\).
\end{proof}

\chapter{Computation of the Weight Vector}
\label{sec:Algorithm}

With the preparation from \hyperref[sec:Opt_Trans]{Chapter \ref*{sec:Opt_Trans}} we are now able to explicitly compute a weight vector \(w^*\) that gives an optimal transport map from \(\mu\) to \(\nu\) via its corresponding Voronoi diagram. We do this by minimizing the convex function \(\Phi\).\\
The obtained transport map allows for computing the Wasserstein distance between the two measures.

\section{Basic Algorithm}
First, we provide the general algorithm. To increase its performance, we will refine it later by decomposing \(\nu\).

\begin{algorithm}[H]
\caption{Computing \(w^*\) by Minimizing \(\Phi\)}
\begin{algorithmic}[1]
\State \(k\defeq 0\)
\State \(w^0\defeq (0, \dots, 0)\)
	\While{\(\nor{\nabla\Phi(w^k)}_1 > \varepsilon\)}
		\State \(k\defeq k+1\)
		\State apply one convex optimization step on \(\Phi\) starting at \(w^{k-1}\)
		\State assign the obtained vector to \(w^k\)
	\EndWhile
\end{algorithmic}
\end{algorithm}

For the stopping criterion we use the \(\ell_1\) norm:
\[\nor{\nabla\Phi(w^k)}_1 = \sum_i \abs{-\lambda_i+\mu\left(\Vor_S^{w^k}(i)\right)} = \sum_i \abs{-\nu(s^i) + \left(T_S^{w^k}\right)_\#\mu(s^i)}.\]
The gradient of \(\Phi\) measures for each point \(s^i\) in \(S\) how far away we are --- given the current transport map --- from the mass that should be transported to \(s^i\). The \(\ell_1\) norm sums up those mistransported masses. So by the stopping criterion we only allow a total of \(\varepsilon/2\) of mass being mistransported where the factor \(2\) stems from the fact that each such portion of mass counts twice: once for the point where it is missing and once for the point where it is in surplus.

For performing the convex optimization step on \(\Phi\), there are various methods available. We will go more into detail about them in \hyperref[sec:Implementation]{Chapter \ref*{sec:Implementation}}.

\section{\texorpdfstring{Decomposition of \(\nu\)}{Decomposition of v}}
The initial weight vector \(w^0\) can have huge influence on the number of steps we need until being sufficiently close to the weight vector where the minimum of the function \(\Phi\) is reached.\\
\cite{Mer11} proposed a method to find a good starting point which we adopt here. The idea is to decompose \(\nu\) into \(L+1\) discrete measures \(\nu_0, \dots, \nu_L\) where \(\nu_0 = \nu\) and \(\abs{\supp(\nu_l)} < \abs{\supp(\nu_{l+1})}\), meaning the measures get simpler with increasing index \(l\). We then iteratively compute a weight vector \(w^l\) adapted to \((\mu, \nu_l)\) and use \(w^l\) as a starting point for computing \(w^{l-1}\).

\begin{definition}
Let \((\nu_l)_{l\geq 0}\) be a sequence of discrete probability measures with \(\nu_0 = \nu\).
If for every \(l\) there exists a transport map \(\tau_l\) from \(\nu_l\) to \(\nu_{l+1}\), that is,
\[\nu_{l+1} = (\tau_l)_\# \nu_l,\]
we call \((\nu_l)_{l\geq 0}\) a \emph{decomposition} of \(\nu\).
\end{definition}

If we write the measures of such a decomposition as
\[\nu_l = \sum_{p\in S_l} \lambda_p^l \delta_p,\]
the definition implies that the cardinality of the supports of our measures decreases, \(\abs{S_{l+1}} \leq \abs{S_l}\).\\
Note that this is equivalent to the following: We can arrange the masses \(\lambda_p^l\) of \(\nu_l\) into clusters and assign each of these bijectively to a point \(q\) in \(S_{l+1}\) such that the sum off all masses in a cluster equals the mass \(\lambda_q^{l+1}\) of the assigned point in the measure \(\nu_{l+1}\).

Under additional assumptions on the source measure \(\mu\) and the decomposition \((\nu_l)_{l \geq 0}\) the sequence of weight vectors obtained from \((\nu_l)_{l \geq 0}\) as described above indeed converges towards the desired weight vector. In \cite{Mer11} the following theorem is proven:
\begin{theorem}
\label{thm:decomposition}
Let \(\nu\) and \((\nu_n)_{n\geq 1}\) be discrete probability measures supported on finite sets \(S\) and \((S_n)_{n\geq 1}\) respectively such that \(\lim_{n\to\infty} \Wass(\nu, \nu_n) = 0\). Suppose that
\begin{theoremenum}
	\item the support of \(\rho\) is the closure of a connected open set \(\Omega\) with piecewise \(\mathcal{C}^1\) boundary;
	\item there exists a positive constant \(\delta\) such that \(\rho \geq \delta\) on \(\Omega\);
	\item the support of all the measures \(\nu_l\) is contained in a fixed ball \(B(0, M)\).
\end{theoremenum}
Let additionally
\begin{align*}
w &\text{ be adapted to } (\mu, \nu),\\
w^n &\text{ be adapted to } (\mu, \nu_n)
\end{align*}
such that all of those weight vectors v fulfill
\[\int_{\R^d} \frac{1}{2} \left(\norm{x}^2 - \min_{p\in S} \left(\norm{x-p}^2 - v_p\right)\right) \mu(\diff x) = 0.\]
(Note that for every weight vector v there exists exactly one constant \(u\in\R^n\) such that \(u+v\) fulfills the above equation. Thus, the purpose of this additional condition is to make the adapted weight vectors unique since usually they are only unique up to the addition of a constant.)\\
Then, for every sequence of points \(p^n\in S_n\) with \(\lim_n p^n = p \in S\) one has
\[w_p=\lim_n w_{p^n}^n.\]
\end{theorem}

In \hyperref[ssec:Impl_Decomposition]{Section \ref*{ssec:Impl_Decomposition}} we will give the concrete decomposition that we use in our implementation.

\section{\texorpdfstring{Algorithm with Decomposition of \(\nu\)}{Algorithm with Decomposition of v}}
\label{ssec:Algorithm_Decomposition}

Let us assume we have fixed a decomposition \((\nu_l)_{0\leq l\leq L}\) and associated transport maps \((\tau_l)_{0\leq l\leq L-1}\). For convenience we set \(\tau_L(p)\defeq p\) for every \(p\in S_L\). We further define \(\Phi^l\) as our usual \(\Phi\) but with respect to \((\mu, \nu_l)\) instead of \((\mu, \nu)\).

\begin{algorithm}[H]
\caption{Minimizing \(\Phi\) Using a Decomposition of \(\nu\)}
\label{algo:with_decomposition}
\begin{algorithmic}[1]
\State \(w^{L+1,0}\defeq (0, \dots, 0)\)
\For{\(l\defeq L\) to \(0\)}
	\State \(k\defeq 0\)
	\State \(w_p^{l,0} \defeq w_{\tau_l(p)}^{l+1}\) for every \(p\in S_l\)
	\While{\(\nor{\nabla\Phi^l(w^{l,k})}_1 > \varepsilon\)}
		\State \(k\defeq k+1\)
		\State apply one convex optimization step on \(\Phi^l\) starting at \(w^{l,k-1}\)
		\State assign the obtained vector to \(w^{l,k}\)
	\EndWhile
	\State \(w^{l}\defeq w^{l,k}\)
\EndFor
\end{algorithmic}
\end{algorithm}

The thinking behind the way of computing \(w^{l,0}\) from \(w^{l+1}\) is the following: Assume that \(\tau_l\) is the optimal transport map from \(\nu_{l+1}\) to \(\nu_l\) --- which we actually try to come close to in our implementation. Then we assign to each point \(p\in S_l\) the weight of its nearest neighbor \(q\in S_{l+1}\), refining the Voronoi diagram by splitting the Voronoi cell of \(q\) evenly into cells for the surrounding points in \(S_l\).

\chapter{Implementation}
\label{sec:Implementation}

We implemented the algorithm from \hyperref[ssec:Algorithm_Decomposition]{Section \ref*{ssec:Algorithm_Decomposition}} in C++ to prove the practicability of our theoretical results. In the next sections we will give details about this implementation and about the motivation for the choices we made for its different parts.

\section{Special Case: Comparison of Images}
\label{ssec:Comparison_Images}
The Wasserstein distance can be used as a measure for the similarity of two grayscale images by taking one as the source and the other as the target measure. The mass distributions are given by the distribution of gray in the images and the Wasserstein distance tells us how great the effort of transforming the source into the target image by shifting the grey around is. Our implementation provides a way to compare two images using exactly this approach.

Given two \(M\times N\)-images with values \(c_{1,1}^1, \dots, c_{M,N}^1\) and \(c_{1,1}^2, \dots, c_{M,N}^2\) for the pixels \([0,1)\times [0,1), \dots, [M-1,M)\times [N-1,N)\), the first step to get there is to convert them into a continuous and a discrete measure. Suppose the first image is our source. Then we set \(\rho\) to be constant on each pixel \([i,i+1)\times [j,j+1)\):
\[\rho(x) \defeq c_{i,j}^1 \quad\text{for all } x\in [i-1,i)\times [j-1,j).\]
The discrete target measure is obtained by choosing its support as the centers of the pixels and setting its masses to the corresponding values.\\
To get probability measures, we normalize the values to sum up to \(1\) and in order to ensure comparability with differently sized images, we divide the support of the measures by \(\max(M,N)\) which results in it being contained on the unit square \([0,1]\times [0,1]\).

\section{Convex Optimization}
The core of our algorithm is the minimization of \(\Phi\). Since \(\Phi\) is convex and we know its gradient, we have a multitude of numerical methods at our disposal for this. They mainly follow the same three steps to create a sequence \((w^k)_k\) with \(\lim_k\Phi(w^k) = \inf_{v\in\R^n} \Phi(v)\):
\begin{enumerate}
	\item determine a search direction \(\Delta w^k\);
	\item determine a step size \(t_k\);
	\item set \(w^{k+1} \defeq w^k + t_k \Delta w^k\).
\end{enumerate}
We consider the class of descent algorithms, that is, algorithms where \(\Phi(w^{k+1}) < \Phi(w^k)\) unless we have reached a minimum. This is achieved by choosing a descent direction \(\Delta w^k\) which means that \(\nabla\Phi(w^k) \Delta w^k < 0\). Then, by the Taylor approximation of \(\Phi\), we are guaranteed to find a step size \(t_k\) such that \(\Phi(w^{k+1}) < \Phi(w^k)\) if we just choose \(t_k\) small enough.\\
Common methods are the gradient method where \(\Delta w^k = -\nabla\Phi(w^k)\) and the Newton method with \(\Delta w^k = -\left(\nabla^2\Phi(w^k)\right)^{-1} \nabla\Phi(w^k)\). Since we only know the gradient but not the Hessian, we cannot use the latter one directly but need to rely on quasi-Newton methods where \(\nabla^2\Phi(w^k)\) is approximated using the gradient values of preceding steps. One of the most popular ones is the BFGS algorithm. We used the implementation libLBFGS \cite{libLBFGS} of the limited-memory version L-BFGS \cite{Noc80} which has a space requirement of \(\mathcal{O}(n)\) instead of \(\mathcal{O}(n^2)\) since it works without storing the whole approximate Hessian matrix. This is important because in the last step of \hyperref[algo:with_decomposition]{Algorithm~\ref*{algo:with_decomposition}}, where we minimize \(\Phi^0\), we have \(n=MN\) which equals the whole resolution of the images.\\
We found that the L-BFGS method provides a much faster rate of convergence than the gradient method.

The step size is determined using a backtracking line search algorithm. The optimal solution would be to minimize \(\Phi\) along the ray \({\Set{w^k + t \Delta w^k \given t\geq 0}}\):
\[t_k = \argmin_{t\geq 0} \Phi(w^k + t \Delta w^k)\]
and we try to approximate this value.
A backtracking line search starts with a relatively large step size and then decreases it until a certain criterion is fulfilled. The challenge is to neither choose \(t_k\) too small which causes a slow rate of convergence nor too large to prevent overshooting the optimal step and having to go in the opposite direction later. As our criterion we use the Wolfe conditions \cite{Wol69,Wol71} which ensure that both \(\Phi\) and \(\nabla\Phi\) are decreased sufficiently in the current step.

\section{Integration over the Voronoi Cells}
\label{ssec:Integration}
The computation of \(\Phi\) and its gradient which are needed for the convex optimization as well as the computation of the Wasserstein distance require us to integrate over Voronoi cells. The two types of integrals that appear are:
\begin{align*}
\int_{\Vor_S^w(i)} w_i \mu(\diff x) &= w_i \int_{\Vor_S^w(i)} \rho(x) \diff x\\
\intertext{and}
\int_{\Vor_S^w(i)} \lVert x - s^i\rVert \mu(\diff x) &= \int_{\Vor_S^w(i)} \lVert x - s^i\rVert \rho(x) \diff x
\end{align*}

A general way of integrating a function over a Voronoi cell \(\Vor_S^w(i)\) with a border consisting of \(k\) different hyperbola segments is cutting the cell into \(k\) parts by lines from \(s^i\) to the end points of those segments and integrating over the parts separately. This procedure is composed of several steps where the first one is to determine the affine transformation that moves the hyperbola segment onto the hyperbola \(y = 1/x\). Then one applies the reverse transformation on the function and integrates this transformed function over the area between \((0,0)\) and the two end points of the transformed hyperbola segment. How obtaining this transformation can be achieved is explained in detail in \hyperref[ssec:Drawing]{Section~\ref*{ssec:Drawing}}.

However, to make things less complicated, we used a different approach in our implementation. Suppose the support of \(\rho\) is bounded. Then we can put a rectangle around \(\supp(\rho)\) which we subsequently partition into equi-sized squares \([a_l,a_l+c)\times[b_r,b_r+c),\ l=1, \dots, L;\ r = 1, \dots, R\). We integrate \(\rho\) over those squares to get a new density \(\tilde{\rho}\) which replaces \(\rho\) and which is constant on each square:
\[\tilde{\rho} (x) \defeq \int_{[a_l,a_l+c)\times[b_r,b_r+c)} \rho(z) \diff z \quad\text{for all } x\in [a_l,a_l+c)\times[b_r,b_r+c).\]
Note that in the case where \(\rho\) is given by an image as in our implementation and we set the squares to the original pixels or refinements of them, \(\rho\) and \(\tilde{\rho}\) coincide.\\
We proceed to determine the Voronoi cell that the center of a square lies in and assume that each Voronoi cell is a polygon made up of the squares assigned to it this way. This allows us to easily integrate over a cell since it amounts to integrating over a union of squares. The price we pay is a reduced accuracy but we are able to control the error by the choice of the number of squares. This is especially useful when applying the multiscale approach:  The Voronoi diagram constructed from \(\nu_{l+1}\) contains fewer and therefore bigger cells than the diagram constructed from \(\nu_l\). Thus, we can use a rougher resolution for computing the weight vector adapted to \((\mu, \nu_{l+1})\) than for the one adapted to \((\mu, \nu_l)\) without introducing too much error.

\section{\texorpdfstring{Decomposition of \(\nu\), Lloyd's Algorithm}{Decomposition of v, Lloyd's Algorithm}}
\label{ssec:Impl_Decomposition}
\begin{algorithm}[t]
\caption{Computation of \(\nu_{l+1}\) from \(\nu_l\) Using Lloyd's Algorithm}
\begin{algorithmic}[1]
\State \(N \defeq \abs{S_l} / 5\)
\State \(S_{l+1} \defeq \emptyset\)
\State draw \(q^1, \dots, q^N\) randomly and uniquely from \(S_l\)
\While{\(S_{l+1} \neq \{q^1, \dots, q^N\}\)}
	\State \(S_{l+1} \defeq \{q^1, \dots, q^N\}\)
	\State \(C_1\defeq\emptyset, \dots, C_N\defeq\emptyset\)
	\For{\(p\in S_l\)}
		\State \(i \defeq \argmin_k \norm{p-q^k}\)
		\State add \(p\) to \(C_i\)
	\EndFor
	\For{\(i=1\) to \(N\)}
		\State \(q^i \defeq (0, \dots, 0)\)
		\State \(a \defeq 0\) \Comment{accumulated mass for current cluster}
		\For{every \(p\in C_i\)}
			\State \(q^i \defeq q^i + \lambda_p^l p\)
			\State \(a \defeq a + \lambda_p^l\)
		\EndFor
		\State \(q^i \defeq \frac{1}{a} q^i\)
		\State \(\lambda_i^{l+1} \defeq a\)
	\EndFor
\EndWhile
\end{algorithmic}
\end{algorithm}
In view of \hyperref[thm:decomposition]{Theorem \ref*{thm:decomposition}}, we should choose a decomposition \((\nu_l)_{0\leq l\leq L}\) such that \(\Wass(\nu_l, \nu_{l+1})\) is as small as possible. The task of minimizing the Wasserstein distance between two consecutive measures \(\nu_l, \nu_{l+1}\) amounts to finding an optimal clustering of \(\supp(\nu_l)\) into a given number \(N\defeq \abs{S_{l+1}} \leq \abs{S_l}\) of clusters \(C_1, \dots, C_N\) with assigned points \(q^1, \dots, q^N\) in the sense that we try to minimize
\[D\defeq\sum_{p\in S_l} \lambda_p^l \norm{p - q^p},\]
where by \(q^p\) we denote the point assigned to the cluster to which \(p\) belongs. This is a weighted k-means problem. Since the NP-hard \cite{MNV09} standard k-means problem can be reduced to the weighted one by setting the weights to \(1\), it is NP-hard, too. Therefore, we restrict ourselves to finding a local optimum by applying Lloyd's algorithm. For the number of clusters we choose \(N = \abs{S_{l+1}} \defeq \lfloor\abs{S_l} / 5\rfloor\).\\
We initialize \(S_{l+1}\) with a random sample \(q^1, \dots, q^N\) of \(N\) points from \(S_l\). Then we iterate two steps:
\begin{enumerate}
	\item Minimize \(D\) by assigning each point in \(S_l\) to the nearest cluster center.
	\item Compute the new center \(q^i\) of each cluster by taking the weighted average of the points in it. Compute its weight \(\lambda_i^{l+1}\) by summing up the weights of those points.
\end{enumerate}
The algorithm terminates at a local optimum since at each iteration \(D\) decreases and there is only a finite number of ways to split \(S_l\) between \(N\) clusters.

\section{Construction of the Voronoi Diagram}
The evaluation of \(\Phi\) and its gradient at a given weight vector requires the computation of an additively weighted Voronoi diagram. For this we use the package ``2D Apollonius Graphs'' which is part of the CGAL library \cite{cga}. It does not compute an additively weighted Voronoi diagram directly but its dual graph. If we imagine the Voronoi diagram to be a graph where the hyperbola segments are the edges and their end points are the vertices, then its dual graph is the graph where the set of vertices is given by the cells and two vertices are connected if and only if the associated cells are neighbors.
\begin{figure}[H]
  \begin{subfigure}[b]{0.4\textwidth}
    \includegraphics[width=\textwidth]{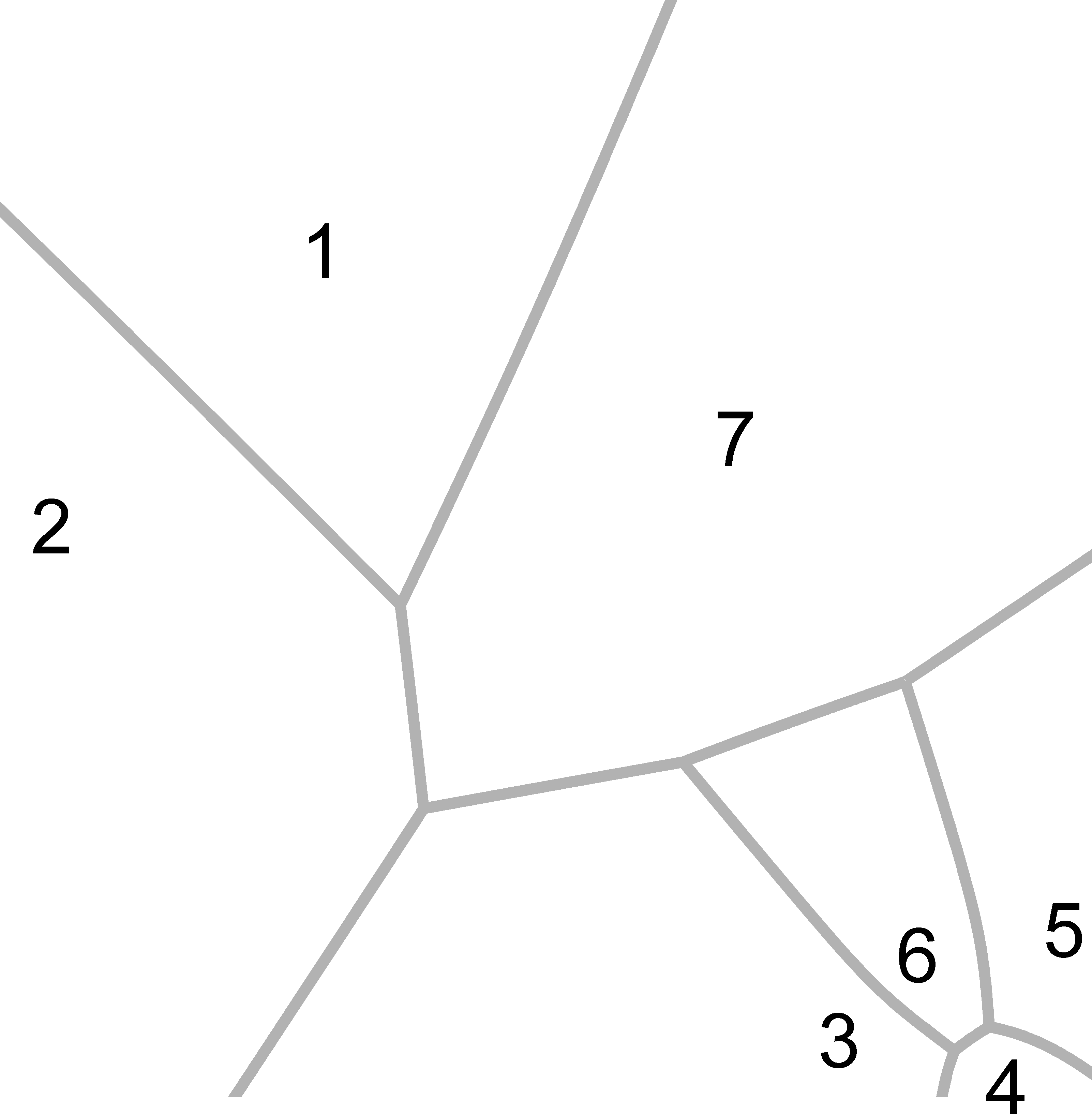}
  \end{subfigure}
  \hfill
	\begin{tikzpicture}[-,auto,node distance=2.7cm,
											thick,main node/.style={font=\sffamily\bfseries}]

		\node[main node] (1) {1};
		\node[main node] (2) [below left of=1,xshift=0.5cm] {2};
		\node[main node] (7) [below right of=1,yshift=0.32cm] {7};
		\node[main node] (3) [below right of=2,xshift=1.9cm,yshift=-0.3cm] {3};
		\node[main node] (6) [below right of=7,xshift=-0.8cm] {6};
		\node[main node] (4) [below of=6,xshift=0.2cm,yshift=1.9cm] {4};
		\node[main node] (5) [right of=6,xshift=-1.9cm,yshift=0.1cm] {5};

		\path[every node/.style={font=\sffamily\small}, line width = 1.2pt]
			(1) edge[color=black!30!white] (2)
					edge[color=black!30!white] (7)
			(2) edge[color=black!30!white] (7)
					edge[color=black!30!white] (3)
			(3) edge[color=black!30!white] (7)
					edge[color=black!30!white] (6)
					edge[color=black!30!white] (4)
			(4) edge[color=black!30!white] (6)
					edge[color=black!30!white] (5)
			(5) edge[color=black!30!white] (6)
					edge[color=black!30!white] (7)
			(6) edge[color=black!30!white] (7);
	\end{tikzpicture}
  \caption{A Voronoi diagram and its dual graph.}
\end{figure}
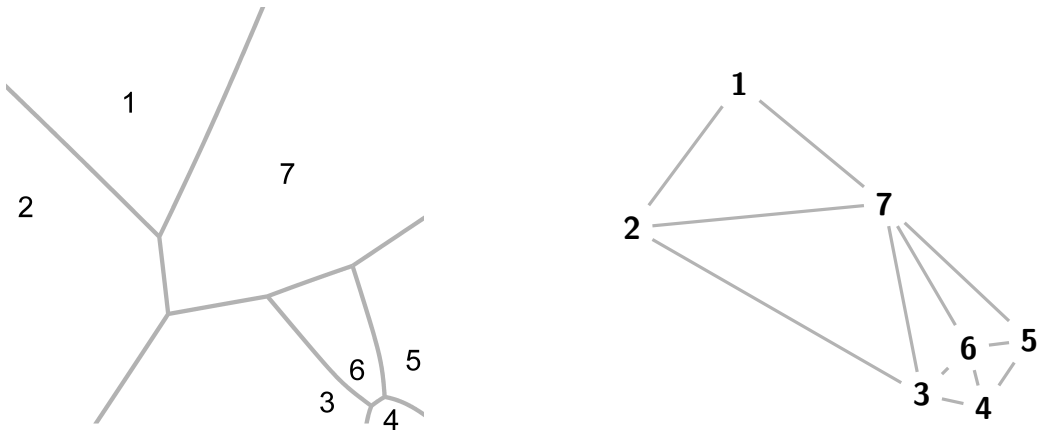

The algorithm is based on \cite{KY02}. Sites are inserted incrementally in three steps:
\begin{enumerate}
	\item Find the site whose circle is closest to the one of the new site. \label{enum:algo_voronoi1}
	\item Decide if the site is trivial, that is, its circle is completely contained in another one and therefore its cell is empty. \label{enum:algo_voronoi2}
	\item Determine the region that gets altered by inserting the site and change it accordingly. \label{enum:algo_voronoi3}
\end{enumerate}

We now describe the steps in detail. In our description, by the distance between two sites we always mean the distance of their circles.\\
Assume that we want to insert a new site \(a\) into the graph.

\begin{enumerate}[wide, labelwidth=!, labelindent=0pt]
	\item We start at a random site \(b\) that is already part of the graph. We consider all neighbors of \(b\). If amongst them we find a \(b'\) which is closer to \(a\) than \(b\), then \(b\) cannot be the nearest neighbor and we continue our search at \(b'\). Otherwise \(b\) is the nearest neighbor. This algorithm has a runtime of \(\mathcal{O}(h)\) where by \(h\) we denote the number of non-trivial sites in the graph.\\
	In the CGAL library it is improved to roughly \(\mathcal{O}(\log h)\) by maintaining a hierarchy of graphs. The lowest level is the original dual graph and each level above contains a small random sample of vertices of the level below, making a search in the graphs at higher levels faster than that in the graphs at lower levels. A search is then performed by applying the algorithm above to the graph at the highest level, continuing at the next lower level where we start at the nearest neighbor from the level above, and descending until we reach the original graph. This approach is quite similar to the decomposition of \(\nu\) whose aim it is to find better starting points, too.\\
	The search for the nearest neighbor is of special interest for us not only in the context of creating the additively weighted Voronoi diagram, but also when determining the Voronoi cell that contains the center of one of the squares that were introduced in \hyperref[ssec:Integration]{Section \ref*{ssec:Integration}}: This cell is the nearest neighbor of the square's center. When using the nearest neighbor search for this purpose, it can be sped up even more. Rather than beginning at a random cell, we start at the cell one of the neighboring squares' center --- whose location we have already determined --- lies in. Often this is already the desired cell or we only have to move few steps to find it since in practice the number of squares exceeds the number of cells by a huge factor and thus their size is much smaller.
	\item After the execution of the first step we know the nearest neighbor of \(a\). Thus, the only thing that is left to do in order to determine whether \(a\) is trivial is to check whether its circle is completely contained in the one of its nearest neighbor. If that is not the case, we continue with \hyperref[enum:algo_voronoi3]{step \ref*{enum:algo_voronoi3}}. Otherwise we can omit \hyperref[enum:algo_voronoi3]{step \ref*{enum:algo_voronoi3}} since inserting \(a\) does not change the Voronoi diagram.
	\item We determine the hyperbola segments that get altered by inserting \(a\) by wandering on them using a depth first search. This is possible since the part of the border of the Voronoi cells being altered is always connected. We start at a segment of the border of the nearest neighbor of \(a\) that gets altered. Next, we check its end points. If an end point gets altered, then we continue on the hyperbola segments adjacent to it, again checking their end points. Once we reach an end point that stays unchanged, we don't need to follow this branch any more and can continue at the next one.\\
	At the end of this procedure we know the hyperbola segments that are affected by inserting \(a\) and can adjust our graph.
\end{enumerate}

The expected time complexity for creating the dual graph of an additively weighted Voronoi diagram of \(n\) sites amongst which \(h\) are non-trivial using the approach above is roughly \(\mathcal{O}(n\log h)\).

\section{Drawing a Voronoi Diagram}
\label{ssec:Drawing}
Drawing the two-dimensional Voronoi diagram for a set of sites amounts to drawing the hyperbola segments each cell's boundary consists of. For a cell to be drawn, two pieces of information are required: the neighboring cells and for each neighbor the end points of the border between this and the original cell. We again use the ``2D Apollonius Graphs'' package from the CGAL library \cite{cga} which can provide both.

\begin{figure}
	\begin{center}
	\includegraphics[width=0.6\textwidth]{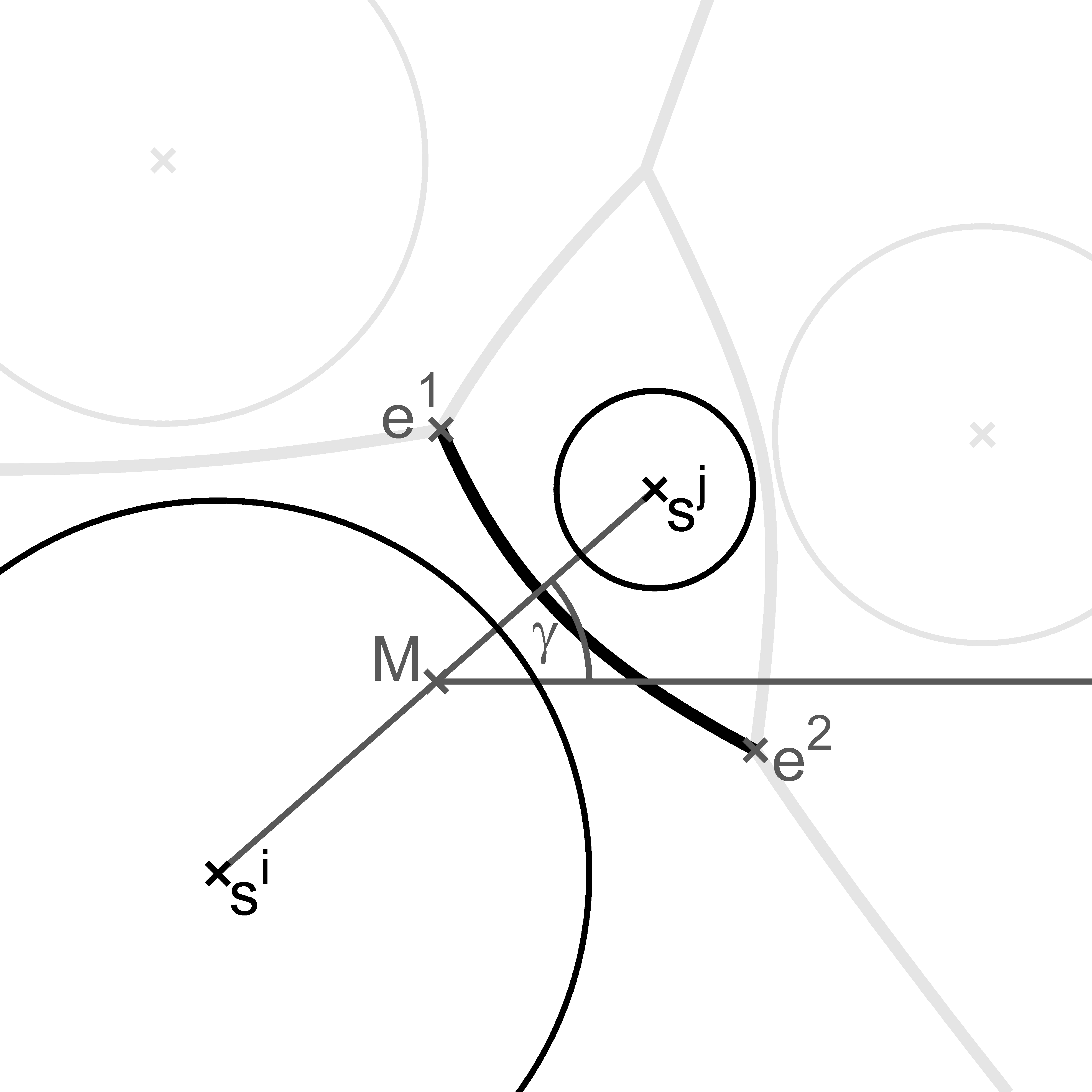}
	\caption{The segment of a hyperbola that we want to draw.}
	\end{center}
\end{figure}

Suppose now we are given a hyperbola segment between sites \(s^i\) and \(s^j\) with end points \(e^1\) and \(e^2\). Our idea is to plot a segment of the function \(x\mapsto 1/x\), \(x>0\), and apply an affine transformation to it to generate the hyperbola segment. In the following we will construct the inverse of this transformation.

We assume that \(s_1^i \leq s_1^j\). Without loss of generality we may further assume that \(w_i > w_j\). If \(w_i < w_j\), we rotate the coordinate system by \(\pi\) and exchange \(i\) and \(j\). If \(w_i=w_j\), the hyperbola degenerates to a line which we can easily draw directly.\\
The branch \(B\) of the hyperbola between \(s^i\) and \(s^j\) is formed by the points \(p\) that fulfill the equation
\[\norm{s^i-p} - \norm{s^j-p} = w_i - w_j.\]
All of our transformations will be applied to \(B\); throughout the description of each one \(B\) won't denote the original branch but the object obtained after the execution of the previous transformations.

Let \(M\defeq \frac{1}{2} (s^i - s^j)\) denote the center between \(s^i\) and \(s^j\). Let further \(\gamma\) be the angle between \(\overrightarrow{s^i s^j}\) and the \(x\)-axis, more specifically, the angle \(\angle R M s^j\) where \(R\defeq (s_1^j, M_2)\).\\
First, we move \(M\) to the origin by subtracting it. We then rotate \(B\) by \(-\gamma\) to let the line between its foci coincide with the x-axis by using the rotation matrix
\[\begin{pmatrix}
\begin{array}{rr}
	\cos(-\gamma)& -\sin(-\gamma)\\
	\sin(-\gamma)& \cos(-\gamma)
\end{array}
\end{pmatrix}\\
=
\begin{pmatrix}
\begin{array}{rr}
	\cos(\gamma)& \sin(\gamma)\\
	-\sin(\gamma)& \cos(\gamma)
\end{array}
\end{pmatrix}.\]
Let \(a\defeq (w_i - w_j)/2\) and \(b\defeq \norm{s^i - s^j}/2\). The equation for \(B\) now becomes
\[\norm{p-(-b,0)} - \norm{p-(b,0)} = 2a.\]
By moving one norm to the other side and taking squares, we get
\[-a^2 + p_1 b = a \norm{p-(b,0)}.\]
Squaring again gives
\[a^4 + b^2 p_1^2 = a^2 p_1^2 + a^2 b^2 + a^2 p_2^2\]
which can further be simplified to
\[1 = \frac{p_1^2}{a^2} - \frac{p_2^2}{b^2 - a^2}.\]
That means that by stretching \(B\) with the matrix
\[\begin{pmatrix}
	\frac{1}{a}& 0\\
	0& \frac{1}{\sqrt{b^2 - a^2}}
\end{pmatrix},\]
the hyperbola branch becomes
\begin{equation}
\label{eq:hyperbola}
1 = p_1^2 - p_2^2 \quad\text{where } p_1 > 0.
\end{equation}
Since \(p_1 = \cosh(t)\), \(p_2 = \sinh(t)\) fulfills \eqref{eq:hyperbola} for every \(t\in\R\) and \(\cosh\) grows continuously from \(1\) to infinity, we can parametrize \(B\) by
\[B=\Set{(\cosh(t), \sinh(t)) \given t\in\R}.\]
The last step consists of rotating \(B\) by \(\pi/4\) and stretching it by \(\sqrt{2}\) which coincides with a multiplication with
\begin{equation}
\label{eq:last_trafo}
\sqrt{2}
\begin{pmatrix}
\begin{array}{rr}
	\cos(\pi/4)& -\sin(\pi/4)\\
	\sin(\pi/4)& \cos(\pi/4)
\end{array}
\end{pmatrix}
=
\begin{pmatrix}
\begin{array}{rr}
	1& -1\\
	1& 1
\end{array}
\end{pmatrix}.
\end{equation}
Now we've got
\[\begin{pmatrix}
\begin{array}{rr}
	1& -1\\
	1& 1
\end{array}
\end{pmatrix}
\begin{pmatrix}
	\cosh(t)\\
	\sinh(t)
\end{pmatrix}
=
\begin{pmatrix}
	\cosh(t) - \sinh(t)\\
	\cosh(t) + \sinh(t)
\end{pmatrix}\]
and
\begin{align*}
(\cosh(t)& - \sinh(t)) (\cosh(t) + \sinh(t))\\
&= \cosh(t)^2 - \sinh(t)^2\\
&= 1.
\end{align*}
The graph of our target function \(x\mapsto 1/x\), \(x>0\), is given by the equation
\[G \defeq \Set{p\in \R_+^2 \given p_1 p_2 = 1}\]
and therefore we have reached our goal to transform \(B\) into this graph.

For summarizing the transformations we utilized in one operation, we multiply the matrices where as a shorthand we write \(c\defeq \sqrt{b^2 - a^2}\):
\begin{align*}
&\begin{pmatrix}
\begin{array}{rr}
	1& -1\\
	1& 1
\end{array}
\end{pmatrix}
\begin{pmatrix}
	\frac{1}{a}& 0\\
	0& \frac{1}{c}
\end{pmatrix}
\begin{pmatrix}
\begin{array}{rr}
	\cos(\gamma)& \sin(\gamma)\\
	-\sin(\gamma)& \cos(\gamma)
\end{array}
\end{pmatrix}\\
&=
\begin{pmatrix}
\begin{array}{rr}
	\frac{1}{a}\cos(\gamma) + \frac{1}{c}\sin(\gamma)& \frac{1}{a}\sin(\gamma) - \frac{1}{c}\cos(\gamma)\\
	\frac{1}{a}\cos(\gamma) - \frac{1}{c}\sin(\gamma)& \frac{1}{a}\sin(\gamma) + \frac{1}{c}\cos(\gamma)
\end{array}
\end{pmatrix}\\
&\mathrel{=\vcentcolon} A^{-1}.
\end{align*}
So we can go from \(B\) to \(G\) by
\begin{equation}
\label{eq:hyperbola_trafo_backward}
G = A^{-1} (B - M)
\end{equation}
and from \(G\) to \(B\) by
\begin{equation}
\label{eq:hyperbola_trafo_forward}
B = A G + M
\end{equation}
where
\[A =
\frac{1}{2}
\begin{pmatrix}
\begin{array}{rr}
	a\cos(\gamma) + c\sin(\gamma)& a\cos(\gamma) - c\sin(\gamma)\\
	a\sin(\gamma) - c\cos(\gamma)& a\sin(\gamma) + c\cos(\gamma)
\end{array}
\end{pmatrix}\]
which can easily be seen by inverting the individual matrices.

Instead of the whole hyperbola branch \(B\) we just want the segment between \(e^1\) and \(e^2\). Therefore, we apply transformation \eqref{eq:hyperbola_trafo_backward} to both \(e^1\) and \(e^2\) to obtain \(\tilde{e}^1\) and \(\tilde{e}^2\) and plot \(x\mapsto 1/x\) only between those two points. We get
\[\tilde{G} \defeq \Set{p\in \R_+^2 \given p_1 p_2 = 1,\ \min(\tilde{e}_1^1, \tilde{e}_1^2) \leq p_1 \leq \max(\tilde{e}_1^1, \tilde{e}_1^2)}\]
and with transformation \eqref{eq:hyperbola_trafo_forward} the desired segment.

We could of course plot \((\cosh(t), \sinh(t))\) instead of \((1, 1/x)\), followed by the appropriate transformation. Then transformation \eqref{eq:last_trafo} could be omitted. Performance is the reason why we still rely on our method: Plotting \((1, 1/x)\) is faster than plotting \((\cosh(t), \sinh(t))\).

\bibliographystyle{alpha}
\bibliography{bachelor_thesis}

\end{document}